\documentclass[11pt, a4paper]{amsart}

\usepackage{amsfonts,amsmath,amssymb, amscd,fullpage}
\usepackage{stmaryrd}
\usepackage[all]{xy}
\usepackage{enumerate} 

\newtheorem{theorem}{Theorem}[section]
\newtheorem{lem}[theorem]{Lemma}
\newtheorem{prop}[theorem]{Proposition}

\newtheorem{example}[theorem]{Example}
\newtheorem{remark}[theorem]{Remark}

\newtheorem{notation}[theorem]{Notation}

\theoremstyle{definition}
\newtheorem{defin}[theorem]{Definition}

\newcommand\pf{\begin{proof}}
\newcommand\epf{\end{proof}}

\newcommand{\N}{\mathbb{N}}

\newcommand{\C}{\mathbb{C}}

\numberwithin{equation}{subsection}

\title{Quantum automorphism group of the lexicographic product of finite regular graphs}

\author{Arthur Chassaniol}
\address{}

\begin{document}

\maketitle

\textbf{Abstract.} We study the quantum automorphism group of the lexicographic product of two finite regular graphs, providing a quantum generalization of Sabidussi's structure theorem on the automorphism group of such a graph.\bigskip\bigskip\smallskip

\section{Introduction}

A quantum permutation group on $n$ points is a compact quantum group acting faithfully on the classical space consisting of $n$ points. The following facts were discovered by Wang \cite{art4}.

\begin{enumerate}
 \item There exists a largest quantum permutation group on $n$ points, now denoted $S_n^+$, and called \textsl{the} quantum permutation groups on $n$ points.
\item The quantum group $S_n^+$ is infinite-dimensional if $n \geq 4$, and hence in particular an infinite compact quantum group can act faithfully on a finite classical space. 
\end{enumerate}

Very soon after Wang's paper \cite{art4}, the representation theory of $S_n^+$ was described  by Banica \cite{art8}: it is similar to the one of $SO(3)$ and can be described using tensor categories of non-crossing partitions. This description, further axiomatized and generalized by Banica-Speicher \cite{bs}, led later to spectacular connections with free probability theory, see e.g. \cite{ks}.

The next natural question was the following one: does $S_n^+$ have many non-classical quantum subgroups, or is it isolated as an infinite quantum group acting faithfully on a finite classical space?

In order to find quantum subgroups of $S_n^+$, the quantum automorphism group of a finite graph was defined in \cite{art12,art11}. This construction indeed produced many examples of non-classical quantum permutation groups, answering positively to the above question.  
The known results on the computation of quantum symmetry groups of graphs are summarized in \cite{art2}, where the description of the quantum symmetry group of vertex-transitive graphs of small order (up to $11$) is given (with an exception for the Petersen graph, whose quantum automorphism group remains mysterious).

\medskip

The present paper is a contribution to the study of quantum automorphism groups of finite graphs: we study the quantum automorphism group of a lexicographic product of finite regular graphs, for which we generalize the results from \cite{art2}.
The description of the quantum automorphism group of some lexicographic product of finite graphs was, amongst other ingredients, a key step in \cite{art2} in the description of the quantum automorphism group of small graphs. Recall that if $X$, $Y$ are finite graphs,
their lexicographic product is, roughly speaking, obtained by putting a copy
of $X$ at each vertex of $Y$ (see Section 3 for details). There is, in general, a group embedding
$${\rm Aut}(X)\wr{\rm Aut}(Y) \subset {\rm Aut}(X \circ Y) \quad (*)$$ where the group on the left is the wreath product of 
${\rm Aut}(X)$ by ${\rm Aut}(Y)$. A quantum analogue of the above embedding is given in \cite{art2}, using the free wreath product from \cite{art1}
 and a sufficient spectral condition was given to ensure that the quantum analogue of the embedding is an isomorphism.
However, there exist (vertex-transitive) graphs of order $\geq 12$ that do not satisfy the spectral assumption, and for which the embedding $(*)$ is an isomorphism (see Example \ref{Ex1}), hence the results in \cite{art2} are not sufficient to fully understand quantum symmetry groups of lexicographic products.

A necessary and sufficient condition on the graphs $X$, $Y$
in order that the embedding $(*)$ be an isomorphism was given by Sabidussi in \cite{art9} (see Section 4). The conditions look slightly technical at first sight, but are very easy to check in practice.
In this paper we provide a quantum generalization of Sabidussi's result: we show that for a pair of regular graphs
$X$, $Y$, the quantum analogue of the embedding $(*)$ is an isomorphism if and only if the graphs satisfy Sabidussi's conditions: see Theorem 3.5. Our result covers many graphs that do not satisfy the spectral conditions from \cite{art2}.

As a final comment, we wish to point out that our result, which expresses certain quantum automorphism groups of finite graphs as free wreath products, will be useful to study the representation theory and operator algebraic properties of these quantum groups, thanks to general results on quantum groups obtained as free wreath product     recently proved  in \cite{le, leta, wahl}.

\medskip

 The paper is organized as follows. Section 2 and 3 are preliminary sections: we recall some basic facts  about compact
quantum groups, quantum permutation groups, free wreath products and quantum automorphism groups of finite graphs. 
Section 4 is devoted to quantum automorphism groups of lexicographic product of finite graphs: we state our main result (Theorem 3.5) and prove it  taking for granted a technical result that we call the key lemma. We also examine some examples.
 The final Section 5 is devoted to the proof of the  key lemma.

\section{Compact quantum groups and free wreath product}

We first recall some basic facts concerning compact quantum groups. The books \cite{art14,art16} are convenient references for this topic, and all the definitions we omit can be found there. All algebras in this paper will be unital as well as all algebra morphisms, and $\otimes$ will denote the minimal tensor product of $C^*$-algebras as well as the algebraic tensor product; this should cause no confusion.

\begin{defin} A \emph{Woronowicz algebra} is a $C^*$-algebra $A$ endowed with a $*$-morphism $\Delta:A\to A\otimes A$ satisfying the coassociativity condition and the cancellation law
$$\overline{\Delta(A)(A\otimes 1)}=A\otimes A=\overline{\Delta(A)(1\otimes A)}$$
The morphism $\Delta$ is called the comultiplication of $A$.
\end{defin}

The category of Woronowicz algebras is defined in the obvious way. A commutative Woronowicz algebra is isomorphic with $C(G)$, the algebra of continuous functions on a compact group $G$, unique up to isomorphism, and the category of \emph{compact quantum groups} is defined to be the category dual to the category of Woronowicz algebras. Hence to any Woronowicz algebra $A$ corresponds a unique compact quantum group $G$ according to the heuristic notation $A=C(G)$.

Woronowicz's original definition for matrix compact quantum groups \cite{art7} is still the most useful to produce concrete examples, and we have the following fundamental result \cite{art13}.

\begin{theorem} Let $A$ be a $C^*$-algebra endowed with a $*$-morphism $\Delta: A\to A\otimes A$. Then $A$ is a Woronowicz algebra if and only if there exists a family of unitary matrices $(u_{\lambda})_{\lambda\in\Lambda}\in M_{d_{\lambda}}(A)$ satisfying the following three conditions:
\begin{enumerate}
\item The $*$-subalgebra $A_0$ generated by the entries $u_{ij}^{\lambda}$ of the matrices $(u^{\lambda})_{\lambda\in\Lambda}$ is dense in $A$.
\item For $\lambda\in\Lambda$ and $i,j\in\{1,\dots, d_{\lambda}\}$, one has $\Delta(u_{ij}^{\lambda})=\sum_{k=1}^{d_{\lambda}}{u_{ik}^{\lambda}\otimes u_{kj}^{\lambda}}$.
\item For $\lambda\in\Lambda$, the transpose matrix $(u^{\lambda})^t$ is invertible.
\end{enumerate}
\end{theorem}

In fact, the $*$-algebra $A_0$ in the theorem is canonically defined, and is what is now called a compact Hopf algebra: a Hopf $*$-algebra having all its finite-dimensional comodules equivalent to unitary ones (see \cite{art14,art16}). The counit and antipode of $A_0$, denoted, respectively, $\epsilon$ and $S$, are referred to as the counit and antipode of $A$. The Hopf $*$-algebra $A_0$ is called the \emph{algebra of representation functions} on the compact quantum group $G$ dual to $A$, with another heuristic notation $A_0=\mathcal{O}(G)$. 

Conversely, starting from a compact Hopf algebra, the universal $C^*$-completion yields a Woronowicz algebra in the above sense: see \cite{art14,art16}. In fact there are possibly several different $C^*$-norms on $A_0$, but we will not be concerned with this question.\smallskip

As usual, a (compact) quantum subgroup $H\subset G$ corresponds to a surjective Woronowicz algebra morphism $C(G)\to C(H)$, or to a  surjective Hopf $*$-algebra morphism $\mathcal{O}(G)\to\mathcal{O}(H)$.\smallskip                          

We refer the reader to \cite{art14,art16} for large classes of examples, including $q$-deformations of classical compact Lie groups. In the present paper, we will be interested in the following fundamental example, due to Wang \cite{art4}. First we need some terminology. A matrix $u\in M_n(A)$ is sais to be orthogonal if $u=\bar{u}$ and $uu^t=I_n=u^tu$. A matrix $u$ is said to be magic unitary if all its entries are projections, all distinct elements  of a same row or same column are orthogonal, and sums of rows and columns are equal to $1$. A magic unitary matrix is orthogonal. 

\begin{defin} The $C^*$-algebra $A_s(n)$ is defined to be the universal $C^*$-algebra generated by variables $(u_{ij})_{1\le i,j\le n}$, with relations making $u=(u_{ij})$ a magic unitary matrix.

The $C^*$-algebra $A_s(n)$ admits a Woronowicz algebra structure given by
$$\Delta(u_{ij}) =\sum_{k=1}^n{u_{ik}\otimes u_{kj}},\ \ \ \ \ \ \epsilon(u_{ij}) =\delta_{ij},\ \ \ \ \ \ S(u_{ij}) =u_{ji}$$

The associated compact quantum group is denoted by $S_n^+$, i.e.
$$A_s(n)=C(S_n^+)$$
\end{defin}

\begin{defin} A \emph{quantum permutation algebra} is a Woronowicz algebra quotient of $A_s(n)$ for some $n$. Equivalently, it is a Woronowicz algebra generated by the coefficients of a magic unitary matrix.
\end{defin}

We now come to quantum group actions, studied e.g. in \cite{art3}. They correspond to Woronowicz algebra coactions. Recall that if $B$ is a $C^*$-algebra, a (right) \emph{coaction} of Woronowicz algebra $A$ on $B$ is a $*$-homomorphism $\alpha: B\to B\otimes A$ satisfying the coassociativity condition and 
$$\overline{\alpha(B)(1\otimes A)}=B\otimes A$$

Wang has studied quantum groups actions on finite-dimensional $C^*$-algebras in \cite{art4}, where the following result is proved.

\begin{theorem} The Woronowicz algebra $A_s(n)$ is the universal Woronowicz algebra coacting on $\C^n$, and is infinite-dimensional if $n\ge 4$.
\end{theorem}

The coaction is constructed in the following manner. Let $e_1,\dots,e_n$ be the canonical basis of $\C^n$. Then the coaction $\alpha: \C^n\to\C^n\otimes A_s(n)$ is defined by the formula
$$\alpha(e_i)=\sum_{j=1}^n{e_j\otimes u_{ji}}$$

We refer the reader to \cite{art4} for the precise meaning of universality in the theorem, but roughly speaking this means that $S_n^+$ is the largest compact quantum group acting on $n$ points, and deserves to be called the \emph{quantum permutation group on n points}.\smallskip

Equivalently, Wang's theorem states that any Woronowicz algebra coacting faithfully on $\C^n$ is a quotient of the Woronowicz algebra $A_s(n)$, and shows that quantum groups acting on $n$ points correspond to Woronowicz algebra quotient of $A_s(n)$, and hence to quantum permutation algebras. In particular, there is a surjective Woronowicz algebra morphism $A_s(n)\to C(S_n)$, yielding a quantum group embedding $S_n\subset S_n^+$. More directly, the existence of the surjective morphism $A_s(n)\to C(S_n)$ follows from the fact that $C(S_n)$ is the universal commutative $C^*$-algebra generated by the entries of a magic unitary matrix. See \cite{art4} for details.\smallskip

We now present the construction of the free wreath product by the quantum permutation group, from \cite{art1}. First we recall the definition of the wreath product in the classical case.

\begin{defin} Let $G$ and $H$ be two finite groups and $\Omega$ a set with $H$ acting on it.  Let $K$ be the direct product $$K:=\prod_{\omega\in\Omega}{G_{\omega}}$$ of copies of $G_{\omega}:= G$ indexed by $\Omega$. Then the action of $H$ on $\Omega$ extends in a natural way to an action of $H$ on the group $K$ by
$$h.(g_{\omega})=(g_{h^{-1}.\omega}),\ \ \text{for}\ h\in H \text{ and } (g_\omega)\in \prod_{\omega\in\Omega}{G_{\omega}}$$
Then the \emph{wreath product} of $G$ by $H$, denoted by $G\wr_{\omega} H$, is the semidirect product of $K$ by $H$. 
The normal subgroup $K$ of $G\wr_{\omega} H$ is called the base of the wreath product.
\end{defin}

\begin{notation}If $G$ is a finite group and $H$ a subgroup of the permutation group $S_n$, then we simply denote by $G\wr H$ the wreath product $G\wr_{\omega} H$ with $\omega=\{1,\dots,n\}$.
\end{notation}


\begin{defin} Let $n\in\N^*$ and let $A$ be a Woronowicz algebra. The \emph{free wreath product} of $A$ by the quantum permutation algebra $A_s(n)$ is the quotient of the $C^*$-algebra $ A^{*n}*A_s(n)$ by the two-sided ideal generated by the elements
$$\nu_k(a)u_{ki}-u_{ki}\nu_k(a),\ \ \ 1\le i,\ k\le n,\ \ a\in A,$$
where $\nu_i:A\to  A^{*n}$, $1\le i\le n$ are the canonical $*$-homomorphisms.

The corresponding $C^*$-algebra is denoted by $ A*_wA_s(n)$.
\end{defin}


\begin{theorem} The free wreath product $A*_wA_s(n)$ admits a Woronowicz algebra structure, with for $a\in A$ and let $i,j\in\{1,\dots,n\}$,
$$\Delta(u_{ij})=\sum_{k=1}^{n}{u_{ik}\otimes u_{kj}}$$
$$\Delta(\nu_i(a))=\sum_{k=1}^n{(\nu_i\otimes\nu_k)(\Delta_{A}(a)).(u_{ik}\otimes 1)}$$
\end{theorem}

We can naturally extend this notion to the case $A*_w H$, when $H$ is any quantum permutation algebra. If $A$ and $B$ are quantum permutation algebras with respective generating magic unitary matrices $u$ and $v$ of size $p$ and $n$, then $A*_wB$ is also a quantum permutation algebra (quotient of $A_s(np)$) with generating magic unitary matrix given by
$$w=(w_{ia,jb})=(u_{ij}^{(a)}v_{ab})$$
where $u^{(a)}=(u_{ij}^{(a)})$ are copies of $u$ and by definition of the free wreath product we add the commuting relations
$$u_{ij}^{(a)}v_{ab}=v_{ab}u_{ij}^{(a)}$$

Hence the Woronowicz algebra structure is given by 
$$\Delta(u_{ij}^{(a)})=\sum_{(s,k)\in [1,p]\times[1,n]}         {u_{is}^{(a)}v_{ak}\otimes u_{sj}^{(k)}} ,\ \ \epsilon(u_{ij}^{(a)})=\delta_{ij},\ \ S(u_{ij}^{(a)})=  \sum_{b=1}^n{   u_{ji}^{(b)}v_{ba}  }              $$
$$\Delta(v_{ab})= \sum_{s=1}^n         {v_{as}\otimes v_{sb}},\ \ \ \epsilon(v_{ab})=\delta_{ab},\ \ \ S(v_{ab})=v_{ba}   $$

Using the properties of magic unitary matrices we obtain:

$$\Delta(w_{ia,jb})= \sum_{l=1}^n{ \sum_{k=1}^p{w_{ia,kl}\otimes w_{kl,jb}}},\ \ \  \epsilon(w_{ia,jb})=\delta_{ij}\delta_{ab},\ \ \ S(w_{ia,jb})=w_{jb,ia}$$

\section{Quantum automorphism group of finite graphs}

In this section we recall the definition of the quantum automorphism group of a finite graph $X$ using \cite{art11,art12}. We first introduce some notations.\bigskip

For a finite graph $X$ with $n$ vertices, it is convenient to also call $X$ the set of vertices of $X$. The complement graph of $X$ will be denoted by $X^c$. If $i$ and $j$ are two vertices of $X$ we use the notation $i\sim_X j$ when they are connected and $i\not\sim_X j$ when they are not (or simply $i\sim j$ when no confusion can arise). \smallskip

 If $i$ is a vertex of $X$, we denote by $\mathcal{V}_X(i)$ the set of neighbours of $i$ in $X$ and $\mathcal{W}_X(i):=(\mathcal{V}_X(i)\cup\{i\})^c=\mathcal{V}_{X^c}(i)$ (or $\mathcal{V}(i)$ and $\mathcal{W}(i)$). We also use the following notations for the cardinal: $\lambda(i)=|\mathcal{V}(i)|$, $\alpha(i)=|\mathcal{W}(i)|$. We say that $X$ is regular when $\lambda(i)$ does not depend on $i\in X$, in this case the notations $\lambda(X)$ and $\alpha(X)$ make sense.\smallskip

\begin{defin} The \emph{adjacency matrix} of $X$ is the matrix $d_X=(d_{ij})_{1\le i,j\le n}\in M_n(0,1)$ given by $d_{ij}=1$ if $i,j$ are connected by an edge, and $d_{ij}=0$ if not. The value of $d_{ij}$ will also be called the \emph{nature} of the couple $(i,j)$ in $X$.
\end{defin}

The classical automorphism group of $X$ will be denoted by $\mathrm{Aut}(X)$ (this is a subgroup of $S_n$) and we have the following way to characterize its elements.

\begin{prop}\label{00} Identifying $\sigma\in S_n$ to the associated permutation matrix $P_{\sigma}\in M_n(\{0,1\})$, we have:
$$\sigma\in\mathrm{Aut}(X)\Longleftrightarrow d_XP_{\sigma}=P_{\sigma}d_X$$
\end{prop}\smallskip

This characterization in the classical case leads to the following natural definition of the quantum automorphism group of a finite graph, see \cite{art11}.

\begin{defin} Associated to a finite graph $X$ is the quantum permutation algebra
$$A(X)=A_s(n)/\langle d_Xu=ud_X\rangle$$
where $n$ is the number of vertices of $X$.

\end{defin}\smallskip

The quantum automorphism group corresponding to $A(X)$ is the quantum automorphism group of $X$, denoted $\mathbb{G}_X$. In this way we have a commuting diagram of Woronowicz algebras:

\[\xymatrix{
     A_s(n)=C(S_n^+) \ar[rr]\ar[d]&  & A(X)= C(\mathbb{G}_X) \ar[d]\\
      C(S_n) \ar[rr]  & & C(\mathrm{Aut}(X))
    }\]
with the kernel of the right arrow being the commutator ideal of $A(X)$.  \smallskip

\begin{example} For the graph with $n$ vertices and no edges we have $A(X)=A_s(n)$, so $\mathbb{G}_X=S_n^+$. Moreover  we have $A(X^c)= A(X)$,
because $ud_X=d_Xu$ and $ud_{X^c}=d_{X^c}u$ are equivalent when $u$ is magic unitary.

If $X=C_n$ is the $n$-cycle graph one can show that for $n\ne 4$, $A(C_n)$ is commutative, thus $A(C_n)=C(\mathrm{Aut}(C_n))$ and therefore $\mathrm{Aut}(C_n)=\mathbb{G}_{C_n}=\mathcal{D}_n$,
where $\mathcal{D}_n$ is the $n$-dihedral group.
For more examples see \cite{art2}.
\end{example}

We are now interested in different ways to define $A(X)$ with other sets of relations. The following result is from \cite{art1}, we include a proof for the sake of completeness.

\begin{prop}\label{ppp1} Le $X$ be a graph with $n$ vertices and $u=(u_{ij})_{1\le i,j\le n}$ a magic unitary matrix. The following sets of relations are equivalent:
\begin{enumerate}{\setlength
 \item $d_Xu=ud_X$
\item For all $i,j\in[1,n]$, \ \ \ \ $$\sum_{k\in\mathcal{W}_X(i)}{u_{kj}}=\sum_{k\in\mathcal{W}_X(j)}{u_{ik}}$$
\item For all $i,j,k,l\in[1,n]$,
$$\left( i\sim j\ \text{and}\ k\not\sim l\right)\Longrightarrow u_{ik}u_{jl}=0=u_{ki}u_{lj}$$
}
\end{enumerate}
\end{prop}\smallskip

\begin{proof}

\underline{$(1)\Leftrightarrow (2)$}: Using that $d_Xu=ud_X$ is equivalent to $d_{X^c}u=ud_{X^c}$, this is a direct translation.\smallskip

\underline{$(3)\Rightarrow (2)$}: Let $i,j\in[1,n]$,
\begin{eqnarray*} 
\sum_{k\in\mathcal{W}_X(i)}{u_{kj}}=  \sum_{k\in\mathcal{W}_X(i)}{   \left( \sum_{s=1}^n{u_{is}}\right)   u_{kj}}&=&\sum_{k\in\mathcal{W}_X(i)}{  \sum_{s\in\mathcal{W}_X(j)}{u_{is}}  u_{kj}}  \\
&=&\sum_{s\in\mathcal{W}_X(j)}{  u_{is}\left(\sum_{k\in\mathcal{W}_X(i)}{u_{kj}}\right) }  \\
&=&\sum_{s\in\mathcal{W}_X(j)}{  u_{is}\left(\sum_{k=1}^n{u_{kj}}\right) }  \\
&=&\sum_{s\in\mathcal{W}_X(j)}{  u_{is}}  \\
\end{eqnarray*}

\underline{$(2)\Rightarrow (3)$}: Let $i,j,k,l\in[1,n]$, such that $i\sim j\ \text{and}\ k\not\sim l$. If $k\ne l$ we have $k\in\mathcal{W}_X(l)$ and $j\notin\mathcal{W}_X(i)$ hence:
$$u_{ik}u_{jl}=u_{ik}\left(\sum_{s\in\mathcal{W}_X(l)}{  u_{is}  }\right)u_{jl}=  u_{ik}\left(\sum_{s\in\mathcal{W}_X(i)}{  u_{sl}  }\right)u_{jl}=0$$
$$\text{and}\ \ \ \ \ \ u_{ki}u_{lj}=u_{ki}\left(\sum_{s\in\mathcal{W}_X(l)}{  u_{si}  }\right)u_{lj}=  u_{ki}\left(\sum_{s\in\mathcal{W}_X(i)}{  u_{ls}  }\right)u_{lj}=0\ \ \ \ \ \ \ \ \ \ \  $$

If $k=l$, the announced identity is obvious.
\end{proof}

\section{Lexicographic product of graphs}

We now want to study the quantum automorphism group of the lexicographic product of two graphs by using those of these two graphs.

Let $X$ and $Y$ be two finite graphs. Their lexicographic product is obtained by putting a copy of $X$ at each vertex of $Y$.

\begin{defin} The \emph{lexicographic product} $X\circ Y$ has vertex set $X\times Y$ and edges are given by:
$$ (i,\alpha)\sim (j,\beta)\Leftrightarrow (\alpha=\beta \text{ and } i\sim j) \text{  or  } (\alpha\sim  \beta)$$
\end{defin}

The simplest example is $X\circ X_n$, where $X_n$ is the graph having $n$ vertices and no edges: it is the graph consisting of $n$ disjoint copies of $X$, also denoted by $nX$.

For the automorphism group of $X\circ Y$, we always have an inclusion 
\begin{eqnarray*}
\mathrm{Aut}(X)\wr\mathrm{Aut}(Y)&\hookrightarrow &\mathrm{Aut}(X\circ Y)\\
(\sigma^X_1,\sigma^X_2,\dots,\sigma^X_n,\sigma^Y) &\mapsto & (i,j)\mapsto (\sigma^X_j(i),\sigma^Y(j)) \\
\end{eqnarray*}
where $\sigma^X_k\in\mathrm{Aut}(X)$ for all $k$ and $\sigma^Y\in\mathrm{Aut}(Y)$.

Sabidussi in \cite{art9} characterizes the case of equality in the following theorem. First we need some notations. Following \cite{art9}, we define this two subsets of $X^2$:
$$S_X:=\{(x_0,x_1)\in X^2\mid \mathcal{V}(x_0)=\mathcal{V}(x_1)\text{ and }x_0\ne x_1\}$$
$$T_X:=S_{X^c}=\{(x_0,x_1)\in X^2\mid \mathcal{V}(x_0)\cup\{x_0\}=\mathcal{V}(x_1)\cup\{x_1\}\text{ and }x_0\ne x_1\}$$

\begin{theorem}\label{z}Let $X$, $Y$ be two finite graphs:
$$\mathrm{Aut}(X\circ Y)=\mathrm{Aut}(X)\wr\mathrm{Aut}(Y)\Leftrightarrow \left((S_Y\ne\emptyset\Rightarrow\text X\  {  connected })\text{ and }(T_Y\ne\emptyset\Rightarrow X^c\text{ connected})\right)$$

\end{theorem}

We would like to obtain an analogous result in the quantum case. Such a result will use the free wreath product.

Let $X$ and $Y$ be two finite graphs that have respectively $p$ and $n$ vertices. We denote by $u=(u_{ij})_{1\le i,j\le p}$ and $v=(v_{ab})_{1\le a,b\le n}$ the respective generating magic unitary matrices of $A(X)$ and $A(Y)$. The graph $X\circ Y$ has $np$ vertices labeled as follows: $$(i,a),\ \ 1\le i\le p,\ \ 1\le a\le n$$ We order them in the following way:
$$(1,1)<(2,1)<...<(p,1)<(1,2)<...<(p,2)<...<(1,n)<...<(p,n)$$

We denote by $Z=(Z_{ia,jb})_{1\le i,j\le p\atop 1\le a,b\le n}$ the generating magic unitary matrix of $A(X\circ Y)$, which satisfies

$$\Delta(Z_{ia,jb})=  \sum_{1\le k\le p\atop 1\le l\le n}{Z_{ia,kl}\otimes Z_{kl,jb}},\ \ \  \epsilon(Z_{ia,jb})=\delta_{ij}\delta_{ab},\ \ \ S(Z_{ia,jb})=Z_{jb,ia}$$

We have the following result from \cite{art2}.

\begin{prop} We have a surjective morphism of Woronowicz algebras
$$A(X\circ Y)\longrightarrow A(X)*_w A(Y)$$
given by $Z_{ia,jb}\mapsto u_{ij}^{(a)}v_{ab}$.
\end{prop}

When $Y=X_n$ is the $n$ vertices graph with no edges and $X$ is a connected graph, we have (see \cite{art6,art1})

$$A(X\circ X_n)=A(nX)\simeq A(X)*_w A_s(n)$$\smallskip

Moreover, the following result is shown in \cite{art2}.

\begin{theorem}\label{BB} Let $X$, $Y$ be two finite regular graphs with $X$ connected. If the spectra $\{\lambda_i\}$ of $d_X$ and $\{\mu_j\}$ of $d_Y$ satisfy the condition 
$$\{\lambda_1-\lambda_i\mid i\ne 1\}\cap\{-n\mu_j\}=\emptyset$$
where $n$ and $\lambda_1$ are the order and valence of $X$, then $A(X\circ Y)\simeq A(X)*_w A(Y)$.
\end{theorem}\smallskip

The main result of this paper is this following generalization of Sabidussi's result.

\begin{theorem}\label{T1}Let $X$ and $Y$ be two finite regular graphs.

 If $\left[(S_Y\ne\emptyset)\Rightarrow (X\text{  is connected })\ \text{and } (T_Y\ne\emptyset)\Rightarrow (X^c\text{ is connected})\right]$, then we have
$$A(X\circ Y)\simeq A(X)*_wA(Y)$$
\end{theorem}\smallskip

\begin{remark} The above result is in fact an equivalence: if $A(X\circ Y)\simeq A(X)*_wA(Y)$ then this isomorphism induces an isomorphism on the abelianisations of these algebras, so that $\mathrm{Aut}(X\circ Y)\simeq\mathrm{Aut}(X)\wr\mathrm{Aut}(Y)$. Thus Sabidussi's result yields that the above condition is satisfied.
\end{remark}

If $A(X)=C(\mathbb{G}_X)$ and $A(Y)=C(\mathbb{G}_Y)$ we will denote by $\mathbb{G}_X\wr_*\mathbb{G}_Y$ the quantum subgroup of $S_{np}^+$ such that 

$$A(X)*_wA(Y)=C(\mathbb{G}_X\wr_*\mathbb{G}_Y)$$\smallskip
To prove this theorem we use similar ideas to those in the case $Y=X_n$ and $X$ connected in \cite{art1}, but the general structure is much more complicated.\smallskip

 If $L,J\in[1,n]$ and $k\in [1,p]$ we denote
$$P_L^J(k):=\sum_{s=1}^p{Z_{sL,kJ}}$$\smallskip

\begin{lem}[Key lemma]\label{LL} Assume that $X$ and $Y$ satisfy the assumption of Theorem \ref{T1}. Let $L,J\in[1,n]$, then for all $k_1,k_2\in [1,p]$ we have $$P_L^J(k_1)=P_L^J(k_2)$$ Thus $P_L^J(k)$ does not depend on $k$ and we denote $P_L^J:=P_L^J(k)$, for any $k$.
\end{lem}\smallskip

The proof of this lemma will be the purpose of the next section. We now prove Theorem \ref{T1} by admitting this key lemma.\smallskip

\begin{lem}\label{l0}For $L,L',J\in[1,n]$ we have this following relations:
$$P_L^JP_{L'}^J=\delta_{L,L'}P_L^J,\ \ \  P_J^LP_{J}^{L'}=\delta_{L,L'}P_J^L,\ \ \ \sum_{L=1}^{n}{P_L^J}=1,\ \ \ \sum_{L=1}^n{P_J^L}=1$$

$$\Delta(P_L^J)=\sum_{S=1}^{n}{P_L^S\otimes P_S^J},\ \ \ \ \ \ \epsilon(P_L^J)=\delta_{L,J},\ \ \ \ \ \ \ S(P_L^J)=P_J^L $$
\end{lem}

\begin{proof} Let $L,L',J\in[1,n]$ with $L\ne L'$ and $k\in [1,p]$. We have:
$$P_L^JP_L^J=\sum_{s_1=1}^p{\sum_{s_2=1}^p{Z_{s_1L,kJ}Z_{s_2L,kJ}}}=\sum_{s=1}^p{Z_{sL,kJ}}=P_L^J \ \ \ \ \ \ $$
$$P_L^JP_{L'}^{J}=\sum_{s_1=1}^p{\sum_{s_2=1}^p{Z_{s_1L,kJ}Z_{s_2L',kJ}}}=0,\ \ \ \text{since}\ L\ne L' \ \ \ \ \ $$
$$ \sum_{L=1}^{n}{P_L^J}=\ \sum_{L=1}^{n}{\sum_{s=1}^p{Z_{sL,kJ}}}=1\ \ \ \ \ \ \ \ \ \ \ \ \ \ \ \ \ \ \ \ \ \ \ \ \ \ \ \ \ \ \ \ \ \ \ \ $$
\begin{eqnarray*} 
\sum_{L=1}^n{P_J^L}=\frac{1}{p}\sum_{L=1}^n{  \sum_{k=1}^p{P_J^L(k)}   }=\frac{1}{p}\sum_{L=1}^n{  \sum_{k=1}^p{     \sum_{s=1}^p {Z_{sJ,kL}}          }   }&=&\frac{1}{p}\sum_{s=1}^p{\left(  \sum_{L=1}^n{     \sum_{k=1}^p {Z_{sJ,kL}}          } \right)  }  \\
&=&\frac{1}{p}\sum_{s=1}^p{1  } =1  \\
\end{eqnarray*}
\begin{eqnarray*} 
\Delta(P_L^J)=  \sum_{s=1}^p{\Delta( Z_{sL,kJ} ) }&=&\sum_{s=1}^p{        \sum_{T=1}^{n}{    \sum_{t=1}^p{Z_{sL,tT}\otimes Z_{tT,kJ}}}         }  \ \ \ \  \ \ \ \ \ \ \ \ \ \ \ \ \ \ \ \ \ \ \ \ \ \ \ \\
&=&    \sum_{T=1}^{n}{    \sum_{t=1}^p{    P_L^T\otimes Z_{tT,kJ}        }}          \\
&=&\sum_{T=1}^{n}{      P_L^T\otimes P_T^J       }    \\
\end{eqnarray*}
$$\epsilon(P_L^J)=\sum_{s=1}^p{\epsilon(Z_{sL,kJ})}=\left\lbrace 
\begin{array}{lcl} 
1 & \text{if } L=J\\ 
0 & \text{ otherwise} 
\end{array}\right. \ \ \ \ \ \ \ \ \ \ \ \ \ \ \ \ \ \ \ \ \ \ \ \ \ \ \ \ \ \  \ \ \ \ \ \ \ \ \ \ \ \ \ \ \ $$
\begin{eqnarray*} 
S(P_L^J)=S\left(   \frac{1}{p} \sum_{k=1}^p{P_L^J(k)}     \right)=S\left(\frac{1}{p}  \sum_{k=1}^p{     \sum_{s=1}^p {Z_{sL,kJ}}             }\right) &=&\frac{1}{p}\sum_{s=1}^p{    \sum_{k=1}^p {S(Z_{sL,kJ})}            }  \\
&=&\frac{1}{p}\sum_{s=1}^p{    \sum_{k=1}^p {Z_{kJ,sL}}            }  \\
&=&\frac{1}{p}\sum_{s=1}^p{    P_J^L(s)          }  =P_J^L\\
\end{eqnarray*}

Finally, we have $P_J^LP_{J}^{L'}=S(P_{L'}^{J}P_L^J)=\delta_{L,L'}S(P_L^J)=\delta_{L,L'}P_J^L$, and this finishes the proof.
\end{proof}

\begin{prop}\label{l2} The matrix $P=(P_L^J)_{1\le L,J\le n}$ is magic unitary and commutes with $d_Y$.
\end{prop}\smallskip

\begin{proof}  Lemma \ref{l0} says that $P$ is magic unitary. Let $L,L',J,J'\in[1,n]$ such that $L\sim_Y L'$ and $J\not\sim_Y J'$ then
$$P_L^JP_{L'}^{J'}=\sum_{s=1}^p{  \sum_{s'=1}^p  {  Z_{sL,1J}Z_{s'L',1J'}    } }=0$$
and $P_J^LP_{J'}^{L'}=S(P_{L'}^{J'}P_L^J)=0$, so by Proposition \ref{ppp1} we get $Pd_Y=d_YP$.
\end{proof}

We consider now the matrix $x^{(a)}=(x_{ij}^{(a)})_{1\le i,j\le p}$ with
$$x_{ij}^{(a)}=\sum_{L=1}^n{Z_{ia,jL}}$$

We need the following lemma to prove some properties of $x^{(a)}$.

\begin{lem}\label{l00} Let $L,L',J\in[1,n]$, $i,i',j,j'\in[1,p]$, such that $L\ne L'$, then
$$Z_{iL,jJ}Z_{i'L',j'J}=0=Z_{jJ,iL}Z_{j'J,i'L'}$$
\end{lem}

\begin{proof} With the assumption we get
$$Z_{iL,jJ}P_L^J=Z_{iL,jJ}P_L^J(j)= \sum_{s=1}^p{Z_{iL,jJ}Z_{sL,jJ}}=Z_{iL,jJ}=P_L^JZ_{iL,jJ}$$
and
$$Z_{i'L',j'J}P_L^J=Z_{i'L',j'J}P_L^J(j')= \sum_{s=1}^p{Z_{i'L',j'J}Z_{sL,j'J}}=0=P_L^JZ_{i'L',j'J}$$
Hence
$$Z_{iL,jJ}Z_{i'L',j'J}=Z_{iL,jJ}P_L^JZ_{i'L',j'J}=0$$
Finally we obtain the second equality, by applying the antipode $S$ to the first one, since $S(Z_{iL,jJ})=Z_{jJ,iL}$.
\end{proof}\smallskip

\begin{lem}\label{X0} Let $a,b\in[1,n]$, for $i,i',j\in[1,p]$, we have this following relations:
$$x_{ij}^{(a)}x_{ij'}^{(a)}=\delta_{j,j'}x_{ij}^{(a)},\ \ \  x_{ji}^{(a)}x_{j'i}^{(a)}=\delta_{j,j'}x_{ji}^{(a)},\ \ \ \sum_{k=1}^{n}{x_{ik}^{(a)}}=1,\ \ \ \sum_{k=1}^n{x_{kj}^{(a)}}=1$$

$$\Delta(x_{ij}^{(a)})=\sum_{(s,k)\in [1,p]\times[1,n]}         {x_{is}^{(a)}P_a^k\otimes x_{sj}^{(k)}} ,\ \ \ \epsilon(x_{ij}^{(a)})=\delta_{ij},\ \ \ S(x_{ij}^{(a)})=  \sum_{L=1}^n{   x_{ji}^{(L)}P_L^a  }  $$
and $$x_{ij}^{(a)}P_a^b=P_a^bx_{ij}^{(a)}$$
\end{lem}

\begin{proof}
Using that $Z$ is magic unitary and Lemma \ref{l00} we have:
$$x_{ij}^{(a)}x_{ij'}^{(a)}=\sum_{L=1}^n{      \sum_{L'=1}^n{Z_{ia,jL} Z_{ia,j'L'}}        }=\delta_{j,j'}\sum_{L=1}^n{Z_{ia,jL}}= \delta_{j,j'} x_{ij}^{(a)} $$
$$x_{ji}^{(a)}x_{j'i}^{(a)}=\sum_{L=1}^n{      \sum_{L'=1}^n{Z_{ja,iL} Z_{j'a,iL'}}        }=    \sum_{L=1}^n{Z_{ja,iL} Z_{j'a,iL}}   =       \delta_{j,j'}x_{ji}^{(a)}$$
$$\sum_{k=1}^p{x_{ik}^{(a)}  } =\sum_{k=1}^p{ \sum_{L=1}^n{Z_{ia,kL}} }=1\ \ \ \ \ \ \ \ \ \ \ \ \ \ \ \ \ \ \ \ \ \ \ \ \ \ \ \ \ \ \ \ \ \ \ \ \ \ \ \ \ $$

and with Lemma \ref{l0}, we obtain
$$\sum_{k=1}^p{x_{kj}^{(a)}  }   =\sum_{k=1}^p{   \sum_{L=1}^n{Z_{ka,jL}}    }= \sum_{L=1}^n{ P_a^L(j)   } =\sum_{L=1}^n{ P_a^L}=1$$
\smallskip

We also have, using Lemma \ref{l00},
$$x_{ij}^{(a)}P_a^b=\sum_{L=1}^n{\sum_{s=1}^p Z_{ia,jL}Z_{sa,jb}  }=\sum_{s=1}^p Z_{ia,jb}Z_{sa,jb}=Z_{ia,jb}$$
$$P_a^bx_{ij}^{(a)}=\sum_{L=1}^n{\sum_{s=1}^p Z_{sa,jb}Z_{ia,jL}  }=\sum_{s=1}^p Z_{sa,jb}Z_{ia,jb}=Z_{ia,jb}$$

Then we get
\begin{eqnarray*} 
\Delta(x_{ij}^{(a)})=\sum_{L=1}^n{   \Delta(Z_{ia,jL})            }&=&\sum_{L=1}^n{   \sum_{(s,k)\in [1,p]\times[1,n]}         {Z_{ia,sk}\otimes Z_{sk,jL}}  }\\
&=&  \sum_{(s,k)\in [1,p]\times[1,n]}         {Z_{ia,sk}\otimes \sum_{L=1}^n{ Z_{sk,jL}}  }\\
&=&\sum_{(s,k)\in [1,p]\times[1,n]}         {x_{is}^{(a)}P_a^k\otimes x_{sj}^{(k)}}\\
\end{eqnarray*}
$$\ \ \ \ \ \ \ \ \ \ \ \ \ \ \ \epsilon(x_{ij}^{(a)})=\sum_{L=1}^n{\epsilon(Z_{ia,jL})}=\delta_{ij} \ \ \ \ \ \ \ \  \ \ \ \ \ \ \ \ \ \ \ \ \  \ \  \ \ \  \ \ \ \ \ \ \ \ \  \ \ \ \ \  \ \ \ \ \ \ \ \ \ \ \ \  \ \ \ \ \ \ \  $$
$$S(x_{ij}^{(a)})=\sum_{L=1}^n{S(Z_{ia,jL})}=\sum_{L=1}^n{Z_{jL,ia}}=\sum_{L=1}^n{  x_{ji}^{(L)}P_L^a   }\ \ \ \ \ \ \ \ \ \ \ \ \ \ \ $$

This finishes the proof.
\end{proof}

\begin{prop}\label{X2} For all $a\in[1,n]$ the matrix $x^{(a)}$ is magic unitary and commutes with $d_X$.
\end{prop}\smallskip

\begin{proof}  Lemma \ref{X0} says that $x^{(a)}$ is magic unitary. Let $i,i',j,j'\in[1,p]$ such that $i\sim_X i'$ and $j\not\sim_X j'$, using
$$(i,L)\sim_{X\circ Y} (i',L)\Leftrightarrow i\sim_X i',\ \ \text{for all } L\in[1,n]$$
and from Lemma \ref{l00}, we obtain

$$x_{ij}^{(a)}x_{i'j'}^{(a)}=\sum_{L=1}^n{  \sum_{L'=1}^n  {  Z_{ia,jL}Z_{i'a,j'L'}    } }= \sum_{L=1}^n  {  Z_{ia,jL}Z_{i'a,j'L}    } =0$$
$$x_{ji}^{(a)}x_{j'i'}^{(a)}=\sum_{L=1}^n{     \sum_{L'=1}^n  {Z_{ja,iL}Z_{j'a,i'L'}     }          }=\sum_{L=1}^n{     Z_{ja,iL}Z_{j'a,i'L}             }=0$$
so by Proposition \ref{ppp1} we get $x^{(a)}d_X=d_Xx^{(a)}$.
\end{proof}

We are now ready to prove Theorem \ref{T1} by showing that the surjective morphism of Woronowicz algebras
\begin{eqnarray*}
\Phi: A(X\circ Y)&\to &A(X)*_wA(Y)\\
Z_{ia,jb} &\mapsto & w_{ia,jb}=u_{ij}^{(a)}v_{ab} \\
\end{eqnarray*}
 is an isomorphism.\medskip

\begin{proof}[Proof of Theorem \ref{T1}]
By Propositions \ref{l2} and \ref{X2}, the matrices $P$ and $x^{(a)}$ are magic unitary and commute respectively with $d_Y$ and $d_X$. By Lemma \ref{X0} we know that for all $i,j\in[1,p]$, $a,b\in[1,n]$, $x_{ij}^{(a)}$ and $P_a^b$ commute. This gives us a Woronowicz algebra morphism


\begin{eqnarray*}
\pi: A(X)*_wA(Y)&\to &A(X\circ Y)\\
u_{ij}^{(a)} &\mapsto &  x_{ij}^{(a)}\\
v_{ab} &\mapsto & P_a^b\\
\end{eqnarray*}
which is inverse to $\Phi$ since $$x_{ij}^{(a)}P_a^b=Z_{ia,jb},\ \ \ \sum_{b=1}^n{u_{ij}^{(a)}v_{ab}}=u_{ij}^{(a)}\ \ \ \ \ \ \text{and}\ \ \ \ \ \sum_{s=1}^p{u_{s1}^{(a)}v_{ab}}=v_{ab}$$
This concludes the proof.\end{proof}

\begin{example}Since for $n\ge 5$ the graphs $C_n$ and $C_n^c$ are connected, for all regular graphs $Y$ we have
$$A(C_n\circ Y)\simeq C(\mathcal{D}_n)*_w A(Y):=C(\mathcal{D}_n\wr_*\mathbb{G}_Y)$$
\end{example}

\begin{example}\label{Ex1} Consider the graphs $\{ K_n\circ C_{3k}\ \mid n\ge 2,\ k\ge 2\}$, where $K_n$ is the complete graph with $n$ vertices. Their spectra are as follows:
$$Sp(K_n)=\{-1,n-1\},\ \ \ \ \ Sp(C_{3k})=\left\{2\cos\left( \frac{2s\pi}{3k} \right)\ \mid\ s=1,\dots,3k\right\}\supset\{-1\}$$
Hence the graph $K_n\circ C_{3k}$ does not satisfy the assumption of Theorem \ref{BB}, but satisfy those of Theorem \ref{T1} since $K_n$ is connected and $T_{C_n}=\emptyset$ when $n\ge 4$. We obtain
$$A(K_n\circ C_{3k})\simeq A(K_n)*_wA(C_{3k})=A_s(n)*_wC(\mathcal{D}_{3k}):= C(S_n^+\wr_*\mathcal{D}_{3k})$$\smallskip
\end{example}

\section{Proof of the key lemma}\smallskip\smallskip

First we recall the statement of the key lemma.\smallskip

\begin{lem}\label{LLL} Let $X$ and $Y$ be two regular graphs such that:

$$(S_Y\ne\emptyset\Rightarrow X\text{  is connected })\text{ and  }(T_Y\ne\emptyset\Rightarrow X^c\text{ is connected})$$\smallskip

Then for all $L,J\in[1,n]$, we have:
$$\forall k_1,k_2\in [1,p],\ \  P_L^J(k_1)=P_L^J(k_2)$$
where $P_L^J(k)=\sum_{s=1}^p{Z_{sL,kJ}}$. Hence $P_L^J(k)$ does not depend on $k$.
\end{lem}\smallskip

\begin{remark}In the case $Y=X_n$ and $X$ connected in \cite{art1}, the proof is much quicker and in \cite{art2} with spectral assumptions on $X$ and $Y$, the keypoint is the property that the eigenspaces of $d_X$ are preserved by the natural coaction of $A(X)$.
\end{remark}\smallskip

Considering the switch of vertices of $Y$ we just need to prove the case $J=1$. Then we denote $P_L^1(k):=P_L(k)$. We begin with four lemmas to understand the product $P_L(k_1)P_J(k_2)$ for different values of $L,J,k_1$ and $k_2$.

\begin{lem}\label{L2} Let $L,J\in[1,n]$, $L\ne J$ and $k_1,k_2\in[1,p]$, $k_1\ne k_2$ such that $(L,J)_Y$ and $(k_1,k_2)_X$ do not have the same nature. Then we have \smallskip
$$P_L(k_1)P_J(k_2)=0$$

\end{lem}

\begin{proof} $(L,J)_Y$ and $(k_1,k_2)_X$ do not have the same nature, so by definition of $X\circ Y$, for all $u,t\in[1,p]$, $(tL,uJ)$ and $(k_11,k_21)$ do not have the same nature as well. Using this and Proposition \ref{ppp1}, we get

$$P_L(k_1)P_J(k_2)=\sum_{t=1}^p{\sum_{u=1}^p{Z_{tL,k_11}Z_{uJ,k_21}}    }=0$$ 
and this ends the proof.
\end{proof}

\begin{lem}\label{L21} Let $L,J\in[1,n]$, $L\ne J$ and $k_1,k_2\in[1,p]$, $k_1\ne k_2$ such that $(L,J)_Y$ and $(k_1,k_2)_X$ have the same nature. 

\begin{enumerate}[i.]

\item  For all $Q\in\mathcal{W}_Y(L)\cap\mathcal{V}_Y(J)$, we have:

$$\sum_{s\in\mathcal{W}_X(k_1)\cap\mathcal{V}_X(k_2)}{P_L(k_1)P_Q(s)P_J(k_2)}=pP_L(k_1)P_J(k_2)$$

\item If $Q\in\{L,J\}$ we have:

$$\sum_{s\in\mathcal{W}_X(k_1)\cap\mathcal{V}_X(k_2)}{P_L(k_1)P_Q(s)P_J(k_2)}=\left\lbrace 
\begin{array}{lcl} 
\delta_{Q,L}\alpha(X) P_L(k_1)P_J(k_2)& \text{if} &  J\in \mathcal{V}_Y(L)\\ 
\delta_{Q,J}\lambda(X)P_L(k_1)P_J(k_2) & \text{if} & J\in \mathcal{W}_Y(L)
\end{array}\right.$$




\end{enumerate}\smallskip

\end{lem}

\begin{proof} By Proposition \ref{ppp1}, for all $i,j\in[1,p]$ and $I,J\in[1,n]$ we have

$$\sum_{(s,S)\in\mathcal{W}_{X\circ Y}(i,I)}{Z_{sS,jJ}}=\sum_{(s,S)\in\mathcal{W}_{X\circ Y}(j,J)}{Z_{iI,sS}}$$\smallskip

Then for all $(i,Q)\in[1,p]\times[1,n]$ we have 

$$\sum_{(s,S)\in\mathcal{W}_{X\circ Y}(i,Q)}{Z_{sS,k_21}}=\sum_{(s,1)\in\mathcal{W}_{X\circ Y}(k_2,1)}{Z_{iQ,s1}}+\sum_{(s,S)\in\mathcal{W}_{X\circ Y}(k_2,1)\atop S\ne 1}{Z_{iQ,sS}}$$

$$\text{and}\ \ \ \sum_{(s,S)\in\mathcal{W}_{X\circ Y}(i,Q)}{Z_{sS,k_11}}=\sum_{(s,1)\in\mathcal{W}_{X\circ Y}(k_1,1)}{Z_{iQ,s1}}+\sum_{(s,S)\in\mathcal{W}_{X\circ Y}(k_1,1)\atop S\ne 1}{Z_{iQ,sS}}\ \ \ \ \ \ \ \ $$

We have, by definition of $X\circ Y$,
$$\{(s,S)\in\mathcal{W}_{X\circ Y}(k_2,1),\ S\ne 1\}=\{(s,S)\in\mathcal{W}_{X\circ Y}(k_1,1),\ S\ne 1\}$$

so we get
$$\sum_{(s,S)\in\mathcal{W}_{X\circ Y}(i,Q)}{(Z_{sS,k_21}-Z_{sS,k_11})}=\sum_{(s,1)\in\mathcal{W}_{X\circ Y}(k_2,1)\atop }{Z_{iQ,s1}} -\sum_{(s,1)\in\mathcal{W}_{X\circ Y}(k_1,1)}{Z_{iQ,s1}}\ \ \ \ (*)$$

Summing over $i\in[1,p]$ we obtain
$$\sum_{i=1}^p{\left(\sum_{(s,1)\in\mathcal{W}_{X\circ Y}(k_2,1)\atop }{Z_{iQ,s1}} -\sum_{(s,1)\in\mathcal{W}_{X\circ Y}(k_1,1)}{Z_{iQ,s1}}\right)}=\sum_{s\in\mathcal{W}_{X}(k_2) }{P_Q(s)} -\sum_{s\in\mathcal{W}_{X}(k_1)}{P_Q(s)}$$

\begin{eqnarray*} 
\sum_{i=1}^p{\left(\sum_{(s,S)\in\mathcal{W}_{X\circ Y}(i,Q)}{Z_{sS,k_21}}\right)} &=&\sum_{i=1}^p{\left(\sum_{s\in\mathcal{W}_{X}(i)}{Z_{sQ,k_21}}\right)}+\sum_{i=1}^p{\left(\sum_{S\in\mathcal{W}_{Y}(Q)}{\sum_{s=1}^p{Z_{sS,k_21}}}\right)} \\
&=&\sum_{s=1}^p{\left(\sum_{i\in\mathcal{W}_{X}(s)}{Z_{sQ,k_21}}\right)}+\sum_{i=1}^p{\left(\sum_{S\in\mathcal{W}_{Y}(Q)}{  P_S(k_2) }\right)} \\
&=&\sum_{s=1}^p{\alpha(X)Z_{sQ,k_21} }+p\left(\sum_{S\in\mathcal{W}_{Y}(Q)}{  P_S(k_2) }\right) \\
&=&\alpha(X)P_Q(k_2)+p\left(\sum_{S\in\mathcal{W}_{Y}(Q)}{  P_S(k_2) }\right)\\
\end{eqnarray*}
and in the same way:
$$\sum_{i=1}^p{\left(\sum_{(s,S)\in\mathcal{W}_{X\circ Y}(i,Q)}{Z_{sS,k_11}}\right)}=\alpha(X)P_Q(k_1)+p\left(\sum_{S\in\mathcal{W}_{Y}(Q)}{  P_S(k_1) }\right)$$

Using $(*)$, we obtain
$$p\left( \sum_{S\in\mathcal{W}_{Y}(Q)}{  (P_S(k_2)-P_S(k_1)) } \right)+\alpha(X)(P_Q(k_2)-P_Q(k_1))=\sum_{s\in\mathcal{W}_{X}(k_2) }{P_Q(s)} -\sum_{s\in\mathcal{W}_{X}(k_1)}{P_Q(s)}$$

Since for all $S,S',k$, $P_S(k)P_{S'}(k)=\delta_{S,S'}P_S(k)$, by multiplying by $P_L(k_1)$ on the left and by $P_J(k_2)$ on the right the above equality, we have

$$a_{L,J}(Q)P_L(k_1)P_J(k_2)+\sum_{s\in\mathcal{W}_{X}(k_1)}{P_L(k_1)P_Q(s)P_J(k_2)}=
\sum_{s\in\mathcal{W}_{X}(k_2)}{P_L(k_1)P_Q(s)P_J(k_2)}\ \ \ \ \ \ \ \ \ \ \ \ (**)$$
where $a_{L,J}(Q)=(\delta_{J\in\mathcal{W}_Y(Q)}-\delta_{L\in\mathcal{W}_Y(Q)})p+(\delta_{Q,J}-\delta_{Q,L})\alpha(X)$.\bigskip

We now deal separately with the cases $Q\in \mathcal{W}_Y(L)\cap\mathcal{V}_Y(J)$ and $Q\in\{L,J\}$.\medskip

If $Q\in\mathcal{W}_Y(L)\cap\mathcal{V}_Y(J)$, then $a_{L,J}(Q)=-p$ and $(**)$ gives

\begin{eqnarray*} 
pP_L(k_1)P_J(k_2)&=&\sum_{s\in\mathcal{W}_{X}(k_1)}{P_L(k_1)P_Q(s)P_J(k_2)}-\sum_{s\in\mathcal{W}_{X}(k_2)}{P_L(k_1)P_Q(s)P_J(k_2)} \\
&=&\sum_{s\in\mathcal{W}_{X}(k_1)\atop s\in\mathcal{V}_{X}(k_2)}{P_L(k_1)P_Q(s)P_J(k_2)}-\sum_{s\in\mathcal{W}_{X}(k_2)\atop s\in\mathcal{V}_{X}(k_2)}{P_L(k_1)P_Q(s)P_J(k_2)} \ \ \ \ \  (\text{by Lemma }\ref{L2})\\
&=&\sum_{s\in\mathcal{W}_{X}(k_1)\atop s\in\mathcal{V}_{X}(k_2)}{P_L(k_1)P_Q(s)P_J(k_2)}\\
\end{eqnarray*}
which gives $(i)$.

Now assume that $Q\in\{L,J\}$. \smallskip

\begin{itemize}
\item If $Q=L$ and $J\in\mathcal{V}_Y(L)$ then $a_{L,J}(Q)=a_{L,J}(L)=-\alpha(X)$ so from $(**)$ we get

\begin{eqnarray*} 
\alpha(X)P_L(k_1)P_J(k_2)&=&\sum_{s\in\mathcal{W}_{X}(k_1)}{P_L(k_1)P_L(s)P_J(k_2)}-\sum_{s\in\mathcal{W}_{X}(k_2)}{P_L(k_1)P_L(s)P_J(k_2)} \\
&=&\sum_{s\in\mathcal{W}_X(k_1)\cap\mathcal{V}_X(k_2)}{P_L(k_1)P_L(s)P_J(k_2)} \ \ \ \ \ \ \ \ \ \ \ \ \  (\text{by Lemma }\ref{L2})\\
\end{eqnarray*}\smallskip

\item If $Q=L$ and $J\in\mathcal{W}_Y(L)$ then for all $s\in\mathcal{V}_X(k_2)$, $P_Q(s)P_J(k_2)=0$ by Lemma \ref{L2}. Therefore we obtain

$$\sum_{s\in\mathcal{W}_X(k_1)\cap\mathcal{V}_X(k_2)}{P_L(k_1)P_L(s)P_J(k_2)}=0$$\smallskip

\item If $Q=J$ and $J\in\mathcal{W}_Y(L)$, then $a_{L,J}(Q)=a_{L,J}(J)=\alpha(X)-p=-\lambda(X)-1$ and with $(**)$ we obtain

\begin{eqnarray*} 
(\lambda(X)+1)P_L(k_1)P_J(k_2)&=&\sum_{s\in\mathcal{W}_{X}(k_1)}{P_L(k_1)P_J(s)P_J(k_2)}-\sum_{s\in\mathcal{W}_{X}(k_2)}{P_L(k_1)P_J(s)P_J(k_2)}\\
&=&\sum_{s\in\mathcal{W}_{X}(k_1)}{P_L(k_1)P_J(s)P_J(k_2)}-\sum_{s\in\mathcal{W}_{X}(k_2)\atop s\in\mathcal{W}_{X}(k_1)}{P_L(k_1)P_J(s)P_J(k_2)}\\
&=&\sum_{s\in\mathcal{W}_{X}(k_1)\cap(\mathcal{V}_{X}(k_2)\cup\{k_2\} )}{P_L(k_1)P_J(s)P_J(k_2)}\\
&=&P_L(k_1)P_J(k_2)+\sum_{s\in\mathcal{W}_{X}(k_1)\cap\mathcal{V}_{X}(k_2)}{P_L(k_1)P_J(s)P_J(k_2)}\\
\end{eqnarray*}
which gives
$$\sum_{s\in\mathcal{W}_{X}(k_1)\cap\mathcal{V}_{X}(k_2)}{P_L(k_1)P_J(s)P_J(k_2)}=\lambda(X)P_L(k_1)P_J(k_2)$$\smallskip

\item Finally if $Q=J$ and $J\in\mathcal{V}_Y(L)$ then for all $s\in\mathcal{W}_X(k_1)$, $P_L(k_1)P_J(s)=0$ by Lemma \ref{L2}. Therefore we obtain

$$\sum_{s\in\mathcal{W}_X(k_1)\cap\mathcal{V}_X(k_2)}{P_L(k_1)P_J(s)P_J(k_2)}=0$$
\end{itemize}
These four equalities prove $(ii)$ and end the proof.
\end{proof}




\begin{lem}\label{L3} Let $L,J\in[1,n]$, $L\ne J$ and $k_1,k_2\in[1,p]$, $k_1\ne k_2$ and put $A=\mathcal{W}_X(k_1)\cap\mathcal{V}_X(k_2)$.
Assume that $\mathcal{W}_Y(L)\cap\mathcal{V}_Y(J)\ne\emptyset$. \smallskip

If $J\in\mathcal{V}_Y(L)$, we have
$$\big(|A|-\alpha(X)-p|\mathcal{W}_Y(L)\cap\mathcal{V}_Y(J)|\big)P_L(k_1)P_J(k_2)=0$$

If $J\in\mathcal{W}_Y(L)$, we have
$$\big(|A|-\lambda(X)-p|\mathcal{W}_Y(L)\cap\mathcal{V}_Y(J)|\big)P_L(k_1)P_J(k_2)=0$$

\end{lem}\smallskip

\begin{proof} If $(L,J)_Y$ and $(k_1,k_2)_X$ do not have the same nature then by Lemma \ref{L2} we get $P_L(k_1)P_J(k_2)=0$. We now assume that they have the same nature. Since $\mathcal{W}_Y(L)\cap\mathcal{V}_Y(J)\ne\emptyset$, if $A=\emptyset$, by $(i)$ of Lemma \ref{L21}, we obtain again $P_L(k_1)P_J(k_2)=0$. Now assume that $A\ne\emptyset$.  \smallskip

For all $s\in A$, we have
\begin{eqnarray*} 
P_L(k_1)P_J(k_2)&=&P_L(k_1)\left(\sum_{Q=1}^n{P_Q(s)}\right)P_J(k_2) \\
&=&P_L(k_1)\left(\sum_{Q\in\mathcal{W}_Y(L)\cap\mathcal{V}_Y(J)\cup\{L\}\cup\{J\}}{P_Q(s)}\right)P_J(k_2)\ \ \ \ \ \ \ \text{(by Lemma \ref{L2})} \\
\end{eqnarray*}

and by summing over $s\in A$ we have
$$|A|P_L(k_1)P_J(k_2)=\sum_{Q\in\mathcal{W}_Y(L)\cap\mathcal{V}_Y(J)\cup\{L\}\cup\{J\}}{     \left( \sum_{s\in A}{P_L(k_1)P_Q(s)P_J(k_2)}      \right)     }$$\smallskip

Finaly using Lemma \ref{L21} we get
$$|A|P_L(k_1)P_J(k_2)=\left\lbrace 
\begin{array}{lcl} 
\alpha(X) P_L(k_1)P_J(k_2)+ p|\mathcal{W}_Y(L)\cap\mathcal{V}_Y(J)|P_L(k_1)P_J(k_2)& \text{if} &J\in \mathcal{V}_Y(L)\\ 
\lambda(X) P_L(k_1)P_J(k_2)+ p|\mathcal{W}_Y(L)\cap\mathcal{V}_Y(J)|P_L(k_1)P_J(k_2) & \text{if} & J\in \mathcal{W}_Y(L)
\end{array}\right.$$
which gives the result.
\end{proof}

\begin{lem}\label{L4} Let $L,J\in[1,n]$, $L\ne J$ and $k_1,k_2\in[1,p]$, $k_1\ne k_2$. Then
$$|\mathcal{W}_Y(L)\cap\mathcal{V}_Y(J)|\ne 0\Longrightarrow P_L(k_1)P_J(k_2)=0$$
In particular,
$$\big(    \big(   k_2\in\mathcal{V}_X(k_1)\ \text{ and }\ T_Y=\emptyset   \big)     \ \text{or}\  \big( k_2\in\mathcal{W}_X(k_1)\ \text{ and }\  S_Y=\emptyset  \big) \big)\Rightarrow P_L(k_1)P_J(k_2)=0$$

\end{lem}\smallskip

\begin{proof} If $(L,J)_Y$ and $(k_1,k_2)_X$ do not have the same nature, this follows from Lemma \ref{L2}. We now assume that they have the same nature and begin with the case $k_2\in\mathcal{V}_X(k_1)$ (and $J\in\mathcal{V}_Y(L)$). By assumption $|\mathcal{W}_Y(L)\cap\mathcal{V}_Y(J)|\ge 1$ so we get $\alpha(X)+p|\mathcal{W}_Y(L)\cap\mathcal{V}_Y(J)|\ge p+\alpha(X)$ and by definition $|\mathcal{W}_X(k_1)\cap\mathcal{V}_X(k_2)|\le p$ so we obtain 

$$\alpha(X)+p|\mathcal{W}_Y(L)\cap\mathcal{V}_Y(J)|-|\mathcal{W}_X(k_1)\cap\mathcal{V}_X(k_2)|\ge \alpha(X)$$\smallskip

If $\alpha(X)>0$, Lemma \ref{L3} gives the result, otherwise $\alpha(X)=0$ (which means $X=K_p$) then $\mathcal{W}_X(k_1)=\emptyset$ and using $(i)$ of Lemma $\ref{L21}$, we can conclude since $\mathcal{W}_Y(L)\cap\mathcal{V}_Y(J)\ne \emptyset$.

Now we deal with the case $k_2\in\mathcal{W}_X(k_1)$ (and $J\in\mathcal{W}_Y(L)$). In the same way as before, we prove that

$$\lambda(X)+p|\mathcal{W}_Y(L)\cap\mathcal{V}_Y(J)|-|\mathcal{W}_X(k_1)\cap\mathcal{V}_X(k_2)|\ge \lambda(X)$$\smallskip

If $\lambda(X)>0$ we conclude by Lemma $\ref{L3}$, otherwise we have $X=X_p$ and $\mathcal{V}_X(k_2)=\emptyset$, so $(i)$ of Lemma \ref{L21} gives us the result.\smallskip

For the last result we just need to prove that when $(L,J)_Y$ and $(k_1,k_2)_X$ have the same nature we have

$$\big(\big(   k_2\in\mathcal{V}_X(k_1)\ \text{ and }\ T_Y=\emptyset   \big)     \ \text{or}\  \big( k_2\in\mathcal{W}_X(k_1)\ \text{ and }\  S_Y=\emptyset  \big)\big)\Longrightarrow|\mathcal{W}_Y(L)\cap\mathcal{V}_Y(J)|\ne 0$$\smallskip

Assume that $|\mathcal{W}_Y(L)\cap\mathcal{V}_Y(J)|=0$ then $\mathcal{V}_Y(J)\subset(\mathcal{W}_Y(L))^c=\mathcal{V}_Y(L)\cup\{L\}$. If $T_Y=\emptyset$ and $k_2\in\mathcal{V}_X(k_1)$ then $J\in\mathcal{V}_Y(L)$ and we get $\mathcal{V}_Y(J)\cup\{J\}=\mathcal{V}_Y(L)\cup\{L\}$ by inclusion and equality of the cardinals, contradiction. In the same way if $S_Y=\emptyset$ and $k_2\in\mathcal{W}_X(k_1)$, then $L\notin\mathcal{V}_Y(J)$ and therefore $\mathcal{V}_Y(J)=\mathcal{V}_Y(L)$, contradiction.
\end{proof}

We are now ready to prove Lemma \ref{LLL}.

\begin{proof} The assumptions on $X$ and $Y$ can be divided into $4$ cases:

\begin{itemize}
\item (1) $X$ is connected and $T_Y=\emptyset$.
\item (2) $X^c$ is connected and $S_Y=\emptyset$.
\item (3) $T_Y=\emptyset$ and $S_Y=\emptyset$.
\item (4) $X$ and $X^c$ are connected with $S_Y\ne\emptyset$ and $T_Y\ne\emptyset$.
\end{itemize}\bigskip

\underline{Case (1)}: Let $k_2\in\mathcal{V}_X(k_1)$, using Lemma \ref{L4} we have

$$P_L(k_1)=P_L(k_1)\left(\sum_{J=1}^n{P_J(k_2)} \right)=P_L(k_1)P_L(k_2)$$and
$$P_L(k_2)=\left(\sum_{J=1}^n{P_J(k_1)} \right)P_L(k_2)=P_L(k_1)P_L(k_2),$$
therefore $P_L(k_1)=P_L(k_2)$ when $k_2$ and $k_1$ are connected in $X$. Since $X$ is connected, it follows that $P_L(k_1)=P_L(k_2)$  for all $k_1,k_2\in[1,p]$.\bigskip

\underline{Case (2)}: We check that $(X\circ Y)^c=X^c\circ Y^c$ and $T_Y=S_{Y^c}$. Then, by using case $(1)$, we get the result since $A(X^c)=A(X)$.\bigskip

\underline{Case (3)}: Let $k_1, k_2\in [1,p]$ with $k_1\ne k_2$. Using Lemma \ref{L4} for all $L,J\in [1,n]$ with $L\ne J$, we have $P_L(k_1)P_J(k_2)=0$. So similarly to case $(1)$, we obtain $P_L(k_1)=P_L(k_2)$.\bigskip

\underline{Case 4}: First we assume $k_2\in\mathcal{V}_X(k_1)$. If $J\in\mathcal{W}_Y(L)$, then by Lemma \ref{L2}, we have $P_L(k_1)P_J(k_2)=0$. If $J\in\mathcal{V}_Y(L)$, by Lemma \ref{L3}, we get

$$\big(|\mathcal{W}_X(k_1)\cap\mathcal{V}_X(k_2)|-\alpha(X)-p|\mathcal{W}_Y(L)\cap\mathcal{V}_Y(J)|\big)P_L(k_1)P_J(k_2)=0$$\smallskip

Then using Lemma \ref{L4} we have $P_L(k_1)P_J(k_2)=0$ or 
$$|\mathcal{W}_X(k_1)\cap\mathcal{V}_X(k_2)|=\alpha(X)\ \  (=|\mathcal{W}_X(k_1)|)$$

This last equality implies that $\mathcal{W}_X(k_1)\subset\mathcal{V}_X(k_2)$. If for all $J\ne L$ we have $P_L(k_1)P_J(k_2)=0$, using the same computation as in case $(1)$ leads to $P_L(k_1)=P_L(k_2)$. Finally we have

$$(*)\ \ k_2\in\mathcal{V}_X(k_1)\Longrightarrow\big(  \big(P_L(k_1)=P_L(k_2)\big) \text{ or } \big(P_L(k_1)\ne P_L(k_2)\text{ and }\mathcal{W}_X(k_1)\subset\mathcal{V}_X(k_2) \big)   \big)$$\smallskip

By taking the complement graph we also obtain

$$(**)\ \ k_2\in\mathcal{W}_X(k_1)\Longrightarrow\big(  \big(P_L(k_1)=P_L(k_2)\big) \text{ or } \big(P_L(k_1)\ne P_L(k_2)\text{ and }\mathcal{V}_X(k_1)\subset\mathcal{W}_X(k_2) \big)   \big)$$\smallskip

We put:

$$S_1=\{ k_2\in\mathcal{V}_X(k_1)\mid P_L(k_1)=P_L(k_2)\}$$
$$S_2=\{ k_2\in\mathcal{V}_X(k_1)\mid P_L(k_1)\ne P_L(k_2)\ \text{and}\ \mathcal{W}_X(k_1)\subset\mathcal{V}_X(k_2)\}$$\smallskip

We have $\mathcal{V}_X(k_1)=S_1\cup S_2\ne\emptyset$ with disjoint union, by $(*)$. \medskip

Assume that $S_2\ne\emptyset$. If all the elements in $S_2$ are connected with all the elements in $S_1$, then $S_2$ ($\ne X$ since $k_1\notin S_2$) is an isolated graph in $X^c$, a contradiction since $X^c$ is connected. Otherwise there exists $s_1\in S_1$ and $s_2\in S_2$ such that $(s_1, s_2)$ is not an edge in $X$, which means that $s_2\in\mathcal{W}_X(s_1)$. By $(**)$ we obtain two cases:

\begin{itemize}
\item $P_L(s_2)=P_L(s_1)=P_L(k_1)$, which is impossible since $s_2\in S_2$
\item $\mathcal{V}_X(s_2)\subset\mathcal{W}_X(s_1)$, which is impossible since $k_1\in\mathcal{V}_X(s_2)$ and $k_1\notin\mathcal{W}_X(s_1)$
\end{itemize}\smallskip

We conclude that $S_2=\emptyset$. Denoting

$$S_1'=\{ k_2\in\mathcal{W}_X(k_1)\mid P_L(k_1)=P_L(k_2)\}$$
$$S_2'=\{ k_2\in\mathcal{W}_X(k_1)\mid P_L(k_1)\ne P_L(k_2)\ \text{and}\ \mathcal{V}_X(k_2)\subset\mathcal{W}_X(k_1)\}$$\smallskip
we prove as above (using that $X$ is connected) that $S_2'=\emptyset$ and we get

$$\forall k_1, k_2\in[1,p],\ \  P_L(k_1)=P_L(k_2)$$
as required.
\end{proof}
\end{document}